\documentclass[12pt,reqno,a4paper]{amsart}%
\usepackage{amsfonts}
\usepackage{amsmath}
\usepackage{amssymb}
\usepackage{mathrsfs}
\usepackage[colorlinks=true,linkcolor=blue,citecolor=red]{hyperref}
\usepackage{graphicx}%
\newtheorem{theorem}{Theorem}[section]
\newtheorem{corollary}[theorem]{Corollary}

\theoremstyle{definition}
\newtheorem{definition}[theorem]{Definition}
\theoremstyle{remark}
\newtheorem{remark}[theorem]{Remark}
\numberwithin{equation}{section} 
\begin{document}
\title[Fekete-Szeg\"{o} inequality for analytic and bi-univalent functions ...]{Fekete-Szeg\"{o} inequality for analytic and bi-univalent functions subordinate to Chebyshev polynomials}

\author{Feras Yousef}
\address{Feras Yousef: Department of Mathematics, Faculty of Science, The University of Jordan, Amman 11942, Jordan.}
\email{fyousef@ju.edu.jo (\textbf{Corresponding author})}
\author{B. A. Frasin}
\address{B. A. Frasin: Faculty of Science, Department of Mathematics, Al al-Bayt University, Mafraq, Jordan.}
\email{bafrasin@yahoo.com}
\author{Tariq Al-Hawary}
\address{Tariq Al-Hawary: Department of Applied Science, Ajloun College, Al-Balqa Applied University,
Ajloun 26816, Jordan.}
\email{tariq\_amh@yahoo.com}
\begin{abstract}
In the present paper, a new subclass of analytic and bi-univalent functions by means of Chebyshev polynomials is introduced. Certain coefficient bounds for functions belong to this subclass are obtained. Furthermore, the Fekete-Szeg\"{o} problem in this subclass is solved.

\end{abstract}
\maketitle

\section{Introduction}

The classical Chebyshev polynomials of degree $n$ of the first and second kinds, which are denoted respectively by $T_{n}(t)$ and $U_{n}(t)$, have generated a great deal of interest in recent years. These orthogonal polynomials, in a real variable $t$ and a complex variable $z$, have played an important role in applied mathematics, numerical analysis and approximation theory. For this reason, Chebyshev polynomials have been studied extensively, see \cite{Doha,Dz,Mason}. In the study of differential equations, the Chebyshev polynomials of the first and second kinds are the solution to the Chebyshev differential equations %
\begin{equation}
(1-t^2)y''-t y'+ n^2 y = 0
\end{equation}
and
\begin{equation}
(1-t^2)y''-3t y'+ n(n + 2)y = 0,
\end{equation}
respectively. The roots of the Chebyshev polynomials of the first kind are used as nodes in polynomial interpolation and the monic Chebyshev polynomials
minimize all norms among monic polynomials of a given degree. For a brief history of Chebyshev polynomials of the first and second kinds and their applications, the reader is referred to \cite{b,S}.

A classical result of Fekete and Szeg\"{o} \cite{Fek} determines the maximum value of $|a_3-\eta a_2^2|$, as a non-linear functional of the real parameter $\eta$, for the class of normalized univalent functions
\begin{equation*}
f(z)=z+a_{2}z^{2}+a_{3}z^{3}+\cdots.
\end{equation*}

There are now several results of this type in the literature, each of them dealing with $|a_3-\eta a_2^2|$ for various classes of functions defined in terms of subordination (see e.g., \cite{Haw,Sriv}). Motivated by the earlier work of Dziok et al. \cite{Dz}, the main focus of this work is to utilize the Chebyshev polynomials expansions to solve Fekete-Szeg\"{o} problem for certain subclass of bi-univalent functions (see, for example, \cite{F0,F100,Bul,tur}).

This paper is divided into three sections with this introduction being the first. In Section 2, we define the class of analytic and bi-univalent functions $\mathscr{B}_{\Sigma }(\lambda ,\mu ,t)$ using the generating function for the Chebyshev polynomials of the second kind, and we also discuss some other definitions and results. Section 3 is devoted to solve problems concerning the coefficients of functions in the class $\mathscr{B}_{\Sigma }(\lambda ,\mu ,t)$. Section 4 is the main part of the paper, we find the sharp bounds of functionals of Fekete-Szeg\"{o} type.

\section{Definitions and preliminaries}
Let $\mathscr{A}$ denote the class of functions of the form:
\begin{equation} \label{ieq1}
f(z)=z+\sum\limits_{n=2}^{\infty}a_{n}z^{n},
\end{equation}
which are \emph{analytic} in the open unit disk $\mathbb{U}$ $=\{z\in \mathbb{C}$ $:$ $\left\vert z\right\vert <1\}$. Further, by $\mathscr{S}$ we shall denote the class of all functions in $\mathscr{A}$ which are \emph{univalent} in $\mathbb{U}$.

Given two functions $f$, $g$ $\in$ $\mathscr{A}$. The function $f(z)$ is said
to be \emph{subordinate} to $g(z)$ in $\mathbb{U}$, written $f(z)$ $\prec$
$g(z)$, if there exists a Schwarz function $\omega(z)$, analytic in $\mathbb{U}$, with
\begin{center}
$\omega(0) = 0$ and $\left\vert \omega(z)\right\vert <1$ for all $z \in\mathbb{U}$,
\end{center}
such that $f(z)=g\left(\omega(z)\right)$ for all $z \in\mathbb{U}$. Furthermore, if the function
$g$ is univalent in $\mathbb{U}$, then we have the following
equivalence (see \cite{mill} and \cite{fer}):%
\[
f(z)\prec g(z)\Leftrightarrow f(0)=g(0)\text{ and }f(\mathbb{U})\subset
g(\mathbb{U}).
\]

The Koebe one-quarter theorem \cite{Duren} asserts that the image of $\mathbb{U}$ under each univalent function $f$ in $\mathscr{S}$ contains a disk of radius $\frac{1}{4}$. According to this, every function $f\in \mathscr{S}$ has an \emph{inverse map} $f^{-1}$, defined by

\begin{center}
$f^{-1}(f(z))=z \ \ \ (z\in \mathbb{U}),$
\end{center}
and

\begin{center}
$f\left(f^{-1}(w)\right)=w$ $\ \ \ \left( |w|<r_{0}(f);r_{0}(f)\geq\frac{1}{4}\right) $.
\end{center}

In fact, the inverse function is given by
\begin{equation} \label{ieq2}
f^{-1}(w)=w-a_{2}w^{2}+(2a_{2}^{2}-a_{3})w^{3}-(5a_{2}^{3}-5a_{2}a_{3}+a_{4})w^{4}+\cdots.
\end{equation}

A function $f\in \mathscr{A}$ is said to be \emph{bi-univalent} in $\mathbb{U}$ if both $f(z)$ and $f^{-1}(w)$ are univalent in $\mathbb{U}$. Let $\Sigma$ denote the class of bi-univalent functions in $\mathbb{U}$ given by (\ref{ieq1}). For a brief history and some intriguing examples of functions and characterization of the class $\Sigma$, see Srivastava et al. \cite{Sri} and Frasin and Aouf \cite{Fra}, see also  \cite{F1,F2,F10,F5,F3,F4}.

The Chebyshev polynomials of the first and second kinds are orthogonal for $t\in \lbrack -1,1]$ and defined as follows:

\begin{definition}
The Chebyshev polynomials of the first kind are defined by the following three-terms recurrence relation:
\begin{eqnarray*}
&& T_0(t)=1, \nonumber \\
&& T_1(t)=t, \nonumber \\
&&T_{n+1}(t):=2tT_n(t)-T_{n-1}(t).
\end{eqnarray*}
\end{definition}

The first few of the Chebyshev polynomials of the first kind are
\begin{equation}
T_{2}(t)=2t^{2}-1, \ T_{3}(t)=4t^{3}-3t, \ T_{4}(t)=8t^{4}-8t^{2}+1,\cdots
\end{equation}

The generating function for the Chebyshev polynomials of the first kind, $T_{n}(t)$, is given by:
\begin{equation*}
F(z, t) = \frac{1-tz}{1-2tz+z^{2}}=\sum\limits_{n=0}^{\infty }T_{n}(t)z^{n} \ \ \ (z\in
\mathbb{U}).
\end{equation*}

\begin{definition}
The Chebyshev polynomials of the second kind are defined by the following three-terms recurrence relation:
\begin{eqnarray*}
&& U_0(t)=1, \nonumber \\
&& U_1(t)=2t, \nonumber \\
&&U_{n+1}(t):=2tU_n(t)-U_{n-1}(t).
\end{eqnarray*}
\end{definition}

The first few of the Chebyshev polynomials of the second kind are
\begin{equation} \label{ieq5}
U_{2}(t)=4t^{2}-1, \ U_{3}(t)=8t^{3}-4t, \ U_{4}(t)=16t^{4}-12t^{2}+1,\cdots
\end{equation}

The generating function for the Chebyshev polynomials of the second kind, $U_{n}(t)$, is given by:
\begin{equation*}
H(z, t) = \frac{1}{1-2tz+z^{2}}=\sum\limits_{n=0}^{\infty }U_{n}(t)z^{n} \ \ \ (z\in
\mathbb{U}).
\end{equation*}

The Chebyshev polynomials of the first and second kinds are connected by the following relations:
\begin{equation*}
\frac{dT_{n}(t)}{dt}%
=nU_{n-1}(t); \ T_{n}(t)=U_{n}(t)-tU_{n-1}(t); \ 2T_{n}(t)=U_{n}(t)-U_{n-2}(t).
\end{equation*}

\begin{definition}
For $\lambda \geq 1,\mu \geq 0$ and $t\in \left( 1/2,1\right) $, a function $%
f\in \Sigma $ given by (\ref{ieq1}) is said to be in the class $\mathscr{B}_{\Sigma }(\lambda ,\mu ,t)$ if the following
subordinations hold for all $z,w\in \mathbb{U}$:
\begin{equation} \label{ieq3}
\left( 1-\lambda \right) \frac{f(z)}{z}+\lambda f^{\prime }(z)+\mu
zf^{\prime \prime }(z)\prec H(z,t):=\frac{1}{1-2tz+z^{2}}
\end{equation}
and
\begin{equation} \label{ieq4}
\ \left( 1-\lambda \right) \frac{g(w)}{w}+\lambda g^{\prime }(w)+\mu
wg^{\prime \prime }(w)\prec H(w,t):=\frac{1}{1-2tw+w^{2}},
\end{equation}
where the function $g(w)=f^{-1}(w)$ is defined by (\ref{ieq2}).
\end{definition}

\begin{remark}
 \begin{enumerate}
       \item {For $\lambda=1$ and $\mu=0$, we have the class $\mathscr{B}_{\Sigma }(1,0,t):=\mathscr{B}_{\Sigma }(t)$ of functions $f\in \Sigma $ given by (\ref{ieq1}) and satisfying the following subordination conditions for all $z,w\in \mathbb{U}$:
\begin{equation*}
f^{\prime }(z) \prec H(z,t)=\frac{1}{1-2tz+z^{2}}
\end{equation*}
and
\begin{equation*}
\ \  g^{\prime }(w) \prec H(w,t)=\frac{1}{1-2tw+w^{2}}.
\end{equation*}
\newline
This class of functions have been introduced and studied by Altinkaya and Yal\c{c}in \cite{F0}.} \vspace{0.13in}
       \item {For $\mu=0$, we have the class $\mathscr{B}_{\Sigma }(\lambda,0,t):=\mathscr{B}_{\Sigma }(\lambda,t)$ of functions $f\in \Sigma $ given by (\ref{ieq1}) and satisfying the following subordination conditions for all $z,w\in \mathbb{U}$:
\begin{equation*}
\left( 1-\lambda \right) \frac{f(z)}{z}+\lambda f^{\prime }(z) \prec H(z,t)=\frac{1}{1-2tz+z^{2}}
\end{equation*}
and
\begin{equation*}
\ \ \left( 1-\lambda \right) \frac{g(w)}{w}+\lambda g^{\prime }(w) \prec H(w,t)=\frac{1}{1-2tw+w^{2}}.
\end{equation*}

This class of functions have been introduced and studied by Bulut et al. \cite{Bul}.}
     \end{enumerate}
\end{remark}

%The Chebyshev polynomials of the first and second kinds are also defined by:

%\begin{center}
%$T_{n}(t)=\cos n\theta $ \ and \ $\displaystyle  U_{n}(t)=\frac{\sin(n+1)\theta }{%
%\sin\theta } \ \ \ (-1<t<1)$,
%\end{center}
%where $n$ denotes the polynomial degree and $t=\cos\theta $.

%Letting $t=\cos\alpha $, where $\alpha \in (-\pi /3,\pi /3)$, someone can obtain
%\begin{equation*}
%\hspace{-1.5in} H(z,t)=\frac{1}{1-2\cos\alpha z+z^{2}}=1+\sum\limits_{n=1}^{\infty }%
%\frac{\sin(n+1)\alpha }{\sin\alpha }z^{n}
%\end{equation*}
%\begin{equation*}
%\hspace{2.5in}=1+2\cos\alpha z+(3\cos^{2}\alpha -\sin^{2}\alpha )z^2+\cdots \ \ (z\in \mathbb{U}).
%\end{equation*}

%Thus, we can write
%\begin{equation*}
%H(z,t)=1+U_{1}(t)z+U_{2}(t)z^{2}+\cdots \ \ \ \left(z\in \mathbb{U},t\in (-1,1)\right),
%\end{equation*}
%where
%\begin{equation} \label{ieq5}
%U_{1}(t)=2t, \ U_{2}(t)=4t^{2}-1, \ U_{3}(t)=8t^{3}-4t, \ U_{4}(t)=16t^{4}-12t^{2}+1,\cdots
%\end{equation}
%are the Chebyshev polynomials of the second kind.

\section{Coefficient bounds for the function class $\mathscr{B}_{\Sigma }(\lambda ,\mu ,t)$}

We begin with the following result involving initial coefficient bounds for the function class $\mathscr{B}_{\Sigma}\left( \lambda ,\mu ,t\right)$.

\begin{theorem}
\label{thm1} Let the function $f(z)$  given by (\ref{ieq1}) be in the class $\mathscr{B}_{\Sigma}\left( \lambda ,\mu ,t\right)$. Then

\begin{equation} \label{theq1}
|a_{2}|\leq \frac{2t\sqrt{2t}}{\sqrt{\left|(1+ \lambda+2\mu)^2-4t^{2}\left[(\lambda+2\mu)^2-2\mu\right]\right|}}
\end{equation}
and
\begin{equation} \label{theq2}
\hspace{-.7in} |a_{3}|\leq \frac{4t^{2}}{(1+ \lambda+2\mu)^2}+\frac{2t}{1+2\lambda+6\mu}.
\end{equation}

\end{theorem}

\begin{proof}

Let $f\in \mathscr{B}_{\Sigma}\left( \lambda ,\mu ,t\right) $. From (\ref{ieq3}) and (\ref{ieq4}), we have

\begin{equation} \label{eq1}
\left( 1-\lambda \right) \frac{f(z)}{z}+\lambda f^{\prime }(z)+\mu
zf^{\prime \prime }(z)=1+U_{1}(t)w(z)+U_{2}(t)w^{2}(z)+\cdots
\end{equation}
and
\begin{equation} \label{eq2}
\qquad \qquad \left( 1-\lambda \right) \frac{g(w)}{w}+\lambda g^{\prime }(w)+\mu
wg^{\prime \prime }(w)=1+U_{1}(t)v(w)+U_{2}(t)v^{2}(w)+\cdots,
\end{equation}
for some analytic functions

\begin{equation*} \label{eq3}
w(z)=c_{1}z+c_{2}z^{2}+c_{3}z^{3}+\cdots \qquad (z\in \mathbb{U}),
\end{equation*}
and
\begin{equation*} \label{eq4}
v(w)=d_{1}w+d_{2}w^{2}+d_{3}w^{3}+\cdots \quad (w\in \mathbb{U}),
\end{equation*}
such that $w(0)=v(0)=0,$ $|w(z)|<1$ $(z\in \mathbb{U})$ and $|v(w)|<1$ $\ (w\in \mathbb{U}).$

It follows from (\ref{eq1}) and (\ref{eq2}) that

\begin{equation*} \label{eq6}
\left( 1-\lambda \right) \frac{f(z)}{z}+\lambda f^{\prime }(z)+\mu
zf^{\prime \prime }(z)=1+U_{1}(t)c_{1}z+\left[U_{1}(t)c_{2}+U_{2}(t)c_{1}^{2}\right] z^{2}+\cdots
\end{equation*}
and
\begin{equation*} \label{eq7}
\quad \left( 1-\lambda \right) \frac{g(w)}{w}+\lambda g^{\prime }(w)+\mu
wg^{\prime \prime }(w)=1+U_{1}(t)d_{1}w+\left[U_{1}(t)d_{2}+U_{2}(t)d_{1}^{2}\right] )w^{2}+\cdots.
\end{equation*}

A short calculation shows that

\begin{equation} \label{eq8}
\left( 1+\lambda +2\mu \right) a_{2}=U_{1}(t)c_{1},
\end{equation}
\begin{equation} \label{eq9}
\left( 1+2\lambda +6\mu \right) a_{3}=U_{1}(t)c_{2}+U_{2}(t)c_{1}^{2},
\end{equation}
and
\begin{equation} \label{eq10}
-\left( 1+\lambda +2\mu \right) a_{2}=U_{1}(t)d_{1},
\end{equation}
\begin{equation} \label{eq11}
\left( 1+2\lambda +6\mu \right)(2a_{2}^{2}-a_{3})=U_{1}(t)d_{2}+U_{2}(t)d_{1}^{2}.
\end{equation}

From (\ref{eq8}) and (\ref{eq10}), we have

\begin{equation} \label{eq12}
c_{1}=-d_{1},
\end{equation}
and
\begin{equation} \label{eq13}
2\left( 1+\lambda +2\mu \right)^2 a_{2}^2=U_{1}^2(t)\left(c_{1}^2+d_{1}^2\right).
\end{equation}

By adding (\ref{eq9}) to (\ref{eq11}), we get

\begin{equation} \label{eq14}
2\left( 1+2\lambda +6\mu \right) a_{2}^{2}=U_{1}(t)\left(c_{2}+d_{2}\right) +U_{2}(t)\left( c_{1}^{2}+d_{1}^{2}\right).
\end{equation}

By using (\ref{eq13}) in (\ref{eq14}), we obtain

\begin{equation} \label{eq15}
\left[ 2\left( 1+2\lambda +6\mu \right) -\frac{2U_{2}(t)}{%
U_{1}^{2}(t)}(1+\lambda +2\mu )^{2}\right] a_{2}^{2}=U_{1}(t)\left(c_{2}+d_{2}\right).
\end{equation}

It is fairly well known \cite{Duren} that if $|w(z)|<1$ and $|v(w)|<1,$ then

\begin{equation} \label{eq5}
|c_{j}|\leq 1  \ \text{and} \ |d_{j}|\leq 1 \ \text{for all} \ j\in \mathbb{N}.
\end{equation}

By considering (\ref{ieq5}) and (\ref{eq5}), we get from (\ref{eq15}) the desired inequality (\ref{theq1}). \\

Next, by subtracting (\ref{eq11}) from (\ref{eq9}), we have

\begin{equation} \label{eq16}
2\left( 1+2\lambda +6\mu \right) a_{3}-2(1+2\lambda +6\mu)a_{2}^{2}=U_{1}(t)\left( c_{2}-d_{2}\right) +U_{2}(t)\left(c_{1}^{2}-d_{1}^{2}\right).
\end{equation}

Further, in view of (\ref{eq12}), it follows from (\ref{eq16}) that

\begin{equation} \label{eq17}
a_{3}=a_{2}^{2}+\frac{U_{1}(t)}{2(1+2\lambda +6\mu )}\left(c_{2}-d_{2}\right).
\end{equation}

By considering (\ref{eq13}) and (\ref{eq5}), we get from (\ref{eq17}) the desired inequality (\ref{theq2}). This completes the proof of Theorem \ref{thm1}.
\end{proof}

Taking $\lambda =1$ and $\mu = 0$ in Theorem \ref{thm1}, we get the following corollary.

\begin{corollary} \label{cor2} \cite{Bul}
Let the function $f(z)$ given by (\ref{ieq1}) be in the class $\mathscr{B}_{\Sigma}\left( t\right)$. Then

\begin{equation*}
\ |a_{2}|\leq \frac{t\sqrt{2t}}{\sqrt{1-t^{2}}},
\end{equation*}
and
\begin{equation*}
|a_{3}|\leq t^{2}+\frac{2}{3}t.
\end{equation*}
\end{corollary}

For Corollary \ref{cor2}, it's worthy to mention that Altinkaya and Yal\c{c}in \cite{F0} have obtained a remarkable result for the coefficient $|a_{2}|$, as shown in the following corollary.

\begin{corollary}
Let the function $f(z)$ given by (\ref{ieq1}) be in the class $\mathscr{B}_{\Sigma}\left( t\right)$. Then

\begin{equation*}
\ |a_{2}|\leq \frac{t\sqrt{2t}}{\sqrt{1+2t-t^{2}}}.
\end{equation*}

\end{corollary}

Taking $\mu = 0$ in Theorem \ref{thm1}, we get the following corollary.

\begin{corollary} \cite{Bul}
Let the function $f(z)$  given by (\ref{ieq1}) be in the class $\mathscr{B}_{\Sigma}\left( \lambda,t\right)$. Then

\begin{equation*}
|a_{2}|\leq \frac{2t\sqrt{2t}}{\sqrt{|(1+ \lambda)^2-4t^{2}\lambda^2|}}
\end{equation*}
and
\begin{equation*}
 |a_{3}|\leq \frac{4t^{2}}{(1+ \lambda)^2}+\frac{2t}{1+2\lambda}.
\end{equation*}
\end{corollary}

\section{Fekete-Szeg\"{o} inequality for the function class $\mathscr{B}_{\Sigma }(\lambda ,\mu ,t)$}
Now, we are ready to find the sharp bounds of Fekete-Szeg\"{o} functional $a_{3}-\eta a_{2}^{2}$ defined for $f \in  \mathscr{B}_{\Sigma}\left( \lambda ,\mu ,t\right)$ given by (\ref{ieq1}).

\begin{theorem}
\label{thm2} Let the function $f(z)$  given by (\ref{ieq1}) be in the class $\mathscr{B}_{\Sigma}\left( \lambda ,\mu ,t\right)$. Then for some $\eta \in \mathbb{R}$,
\begin{equation} \label{theq3}
|a_{3}-\eta a_{2}^{2}|\leq \left\{
\begin{array}{c}
\qquad \ \ \frac{2t}{1+2\lambda+6\mu}, \qquad \ \qquad \ |\eta -1|\leq M \bigskip \\
\frac{8|\eta -1|t^{3}}{\left|(1+\lambda+2\mu)^2-4t^{2}\left[(\lambda+2\mu)^2-2\mu\right]\right|}, \ \ |\eta -1|\geq M%
\end{array}%
\right.
\end{equation}
where
\begin{equation*}
M:= \frac{\left|(1+\lambda+2\mu)^2-4t^{2}\left[(\lambda+2\mu)^2-2\mu\right]\right|}{4(1+2\lambda+6\mu)t^{2}}.
\end{equation*}

\end{theorem}

\begin{proof}

Let $f\in \mathscr{B}_{\Sigma}\left( \lambda ,\mu ,t\right) $. By using (\ref{eq15}) and (\ref{eq17}) for some $\eta \in \mathbb{R} $, we get
\begin{equation*}
 a_{3}-\eta a_{2}^{2}=\left( 1-\eta \right) \left[ \frac{U_{1}^{3}(t)\left(c_{2}+d_{2}\right) }{2(1+2\lambda +6\mu )U_{1}^{2}(t)-2(1+\lambda +2\mu
)^{2}U_{2}(t)}\right] +\frac{U_{1}(t)\left( c_{2}-d_{2}\right) }{2(1+2\lambda +6\mu )}
\end{equation*}
\begin{equation*}
\qquad \ \ =U_{1}(t)\left[ \left( h(\eta )+\frac{1}{2(1+2\lambda +6\mu )}\right)
c_{2}+\left( h(\eta )-\frac{1}{2(1+2\lambda +6\mu )}\right) d_2\right],
\end{equation*}
where

\begin{equation*}
 h(\eta )=\frac{U_{1}^{2}(t)\left( 1-\eta \right) }{2\left[ (1+2\lambda
+6\mu )U_{1}^{2}(t)-(1+\lambda +2\mu )^{2}U_{2}(t)\right]}.
\end{equation*}

Then, in view of (\ref{ieq5}), we easily conclude that

\begin{equation*}
|a_{3}-\eta a_{2}^{2}|\leq \left\{
\begin{array}{c}
\frac{2t}{1+2\lambda +6\mu }, \ \ |h(\eta )|\leq \frac{1}{2\left( 1+2\lambda
+6\mu \right) } \bigskip \\
 4|h(\eta )|t, \ \ |h(\eta )|\geq \frac{1}{2\left( 1+2\lambda +6\mu \right) }%
\end{array}%
\right.
\end{equation*}

This proves Theorem \ref{thm2}.
\end{proof}

We end this section with some corollaries concerning the sharp bounds of Fekete-Szeg\"{o} functional $a_{3}-\eta a_{2}^{2}$ defined for $f \in  \mathscr{B}_{\Sigma}\left( \lambda ,\mu ,t\right)$ given by (\ref{ieq1}).

Taking $\eta =1$ in Theorem \ref{thm2}, we get the following corollary.

\begin{corollary}
Let the function $f(z)$ given by (\ref{ieq1}) be in the class $\mathscr{B}_{\Sigma}\left( \lambda ,\mu ,t\right)$. Then
\begin{equation*}
|a_{3}-a_{2}^{2}|\leq \frac{2t}{1+2\lambda +6\mu}.
\end{equation*}

\end{corollary}

Taking $\lambda =1$ and $\mu = 0$ in Theorem \ref{thm2}, we get the following corollary.

\begin{corollary} \label{corol2}
Let the function $f(z)$ given by (\ref{ieq1}) be in the class $\mathscr{B}_{\Sigma}\left( t\right)$. Then for some $\eta \in \mathbb{R}$,
\begin{equation*}
|a_{3}-\eta a_{2}^{2}|\leq \left\{
\begin{array}{c}
\frac{2}{3}t,\qquad \ |\eta -1|\leq \frac{1-t^{2}}{3t^{2}} \bigskip \\
\frac{2|\eta -1|t^{3}}{1-t^{2}}, \ \ |\eta -1|\geq \frac{1-t^{2}}{3t^{2}}%
\end{array}%
\right.
\end{equation*}

\end{corollary}

Taking $\eta =1$ in Corollary \ref{corol2}, we get the following corollary.

\begin{corollary} %\cite{Bul}
Let the function $f(z)$ given be (\ref{ieq1}) be in the class $\mathscr{B}_{\Sigma}\left( t\right)$. Then
\begin{equation*}
|a_{3}-a_{2}^{2}|\leq \frac{2}{3} t.
\end{equation*}

\end{corollary}

Taking $\mu = 0$ in Theorem \ref{thm2}, we get the following corollary.

\begin{corollary} \label{corol222}
Let the function $f(z)$  given by (\ref{ieq1}) be in the class $\mathscr{B}_{\Sigma}\left( \lambda ,t\right)$. Then for some $\eta \in \mathbb{R}$,
\begin{equation}
|a_{3}-\eta a_{2}^{2}|\leq \left\{
\begin{array}{c}
 \frac{2t}{1+2\lambda}, \quad \qquad \ \ |\eta -1|\leq \frac{\left|(1+\lambda)^2-4t^{2}\lambda^2\right|}{4(1+2\lambda)t^{2}} \bigskip \\
\frac{8|\eta -1|t^{3}}{\left|(1+\lambda)^2-4t^{2}\lambda^2\right|},  \ |\eta -1|\geq \frac{\left|(1+\lambda)^2-4t^{2}\lambda^2 \right|}{4(1+2\lambda)t^{2}}%
\end{array}%
\right.
\end{equation}
\end{corollary}

Taking $\eta =1$ in Corollary \ref{corol222}, we get the following corollary.

\begin{corollary} % \cite{Bul}
Let the function $f(z)$ given by (\ref{ieq1}) be in the class $\mathscr{B}_{\Sigma}\left( \lambda ,t\right)$. Then
\begin{equation*}
|a_{3}-a_{2}^{2}|\leq \frac{2t}{1+2\lambda}.
\end{equation*}

\end{corollary}

%\thanks{{\bf Acknowledgement}: The authors would like to thank the referees for their careful reading and helpful comments.}


\begin{thebibliography}{99}

\bibitem {Haw} T. Al-Hawary, B. A. Frasin and M. Darus, \emph{Fekete-Szeg\"{o} problem for certain classes of analytic functions of complex order defined by the Dziok-Srivastava operator}, Acta Mathematica Vietnamica 39.2 (2014), 185--192.

\bibitem {F1}  \c{S}. Altinkaya and S. Yal\c{c}in, \emph{Initial coefficient bounds for a general class of bi-univalent functions}, International Journal of Analysis, Article ID 867871,(2014), 4 pp.

\bibitem {F2} \c{S}. Altinkaya and S. Yal\c{c}in, \emph{Coefficient bounds for a subclass of bi-univalent functions}, TWMS Journal of Pure and Applied Mathematics, 6 (2) 2015.

\bibitem {F10} \c{S}. Altinkaya and S. Yal\c{c}in, \emph{Coefficient Estimates for Two New Subclasses of Bi-univalent Functions with respect to Symmetric Points}, Journal of Function Spaces, (2015), Article ID 145242, 5 pages.

\bibitem {F0} \c{S}. Altinkaya and S. Yal\c{c}in, \emph{Estimates on coefficients of a general subclass of bi-univalent functions associated with symmetric q-derivative operator by means of the Chebyshev polynomials}, Asia Pacific Journal of Mathematics, 4(2) (2017), 90-99.

\bibitem {F100} \c{S}. Altinkaya and S. Yal\c{c}in, \emph{On the Chebyshev polynomial coefficient problem of some subclasses of bi-univalent functions}, Gulf Journal of Mathematics, (2017), 1464-1473.

\bibitem {Bul} S. Bulut, N. Magesh and V. K. Balaji, \emph{Initial bounds for analytic and bi–univalent functions by means of chebyshev polynomials}, Analysis 11.1 (2017), 83--89.

\bibitem {Doha} E. H. Doha, \emph{The first and second kind Chebyshev coefficients of the moments of the general-order
derivative of an infinitely differentiable function}, Int. J. of Comput. Math. 51 (1994), 21-–35.

\bibitem {Duren} P. L. Duren, \emph{Univalent functions}, Grundlehren der Mathematischen Wissenschaften 259, Springer-Verlag, New York, 1983.

\bibitem {Dz} J. Dziok, R. K. Raina and J. Sokol, \emph{Application of Chebyshev polynomials to classes of analytic functions}, Comptes Rendus Mathematique 353.5 (2015), 433--438.

\bibitem {Fra} B. A. Frasin and M. K. Aouf, \emph{New subclasses of bi-univalent functions}, Appl. Math. Lett. 24, 9 (2011), 1569-–1573.

\bibitem {F5} B. A. Frasin and  Tariq Al-Hawary, \emph{Initial Maclaurin Coefficients Bounds for New Subclasses of Bi-univalent Functions}, Theory and App. of Math. \& Computer Science 5 (2) (2015) 186--193.

\bibitem {Fek} M. Fekete and G. Szeg\"{o}, \emph{Eine Bermerkung \"{u}ber ungerade schlichte Funktionen}, Journal of the London Mathematical Societ 1.2 (1933), 85-–89.

\bibitem {tur} H. \"{O}. G\"{u}ney, G. Murugusundaramoorthy and K. Vijaya, \emph{Coefficient Bounds for Subclasses of Biunivalent Functions Associated with the Chebyshev Polynomials}, Journal of Complex Analysis (2017), 11 pages.

\bibitem {F3} N. Magesh and J. Yamini, \emph{Coefficient bounds for a certain subclass of bi-univalent functions}, International Mathematical Forum, 8 (27) (2013), 1337-1344.

\bibitem {Mason} J. C. Mason, \emph{Chebyshev polynomial approximations for the L-membrane eigenvalue problem}, SIAM J. Appl. Math. 15 (1967), 172-–186.

\bibitem {mill} S. S. Miller and P. T. Mocanu, \emph{Differential Subordination: theory and applications},
CRC Press, New York, 2000.

\bibitem {F4} S. Porwal and M. Darus, \emph{On a new subclass of bi-univalent functions}, J. Egypt. Math. Soc., 21 (3) (2013), 190-193.

\bibitem {b}B. Simon, \emph{Orthogonal polynomials on the unit circle}, American Mathematical Society, 2009.

\bibitem {Sriv} H. M. Srivastava, A. K. Mishra, and M. K. Das, \emph{The fekete-szeg\"{o}-problem for a subclass of close-to-convex functions}, Complex Variables and Elliptic Equations 44.2 (2001), 145--163.

\bibitem {Sri} H. M. Srivastava, A. K. Mishra and P. Gochhayat, \emph{Certain subclasses of analytic and bi-univalent functions}, Appl. Math. Lett. 23 (2010), 1188–-1192.

\bibitem {S}G. Szeg\"{o}, \emph{Orthogonal polynomials}, American Mathematical Society, New York, 2003.

Szeg, Gabor. Orthogonal polynomials. Vol. 23. American Mathematical Soc., 1939.

\bibitem {fer} F. Yousef, A. A. Amourah and M. Darus, \emph{Differential sandwich theorems for p-valent functions associated with a certain generalized differential operator and integral operator}, Italian Journal of Pure and Applied Mathematics 36 (2016), 543--556.





\end{thebibliography}
\end{document}